\documentclass[11pt]{article}
\usepackage[pdftex]{color}
\usepackage[pdftex]{geometry}
\usepackage{latexsym,euscript,amsthm,amsfonts,setspace}
\usepackage[pdftex]{graphicx}
\usepackage{times}
\usepackage{amsmath,amssymb,multicol,amscd}

\newcommand{\DEF}{:=}

\newcommand{\D}{\displaystyle}
\newcommand{\DIS}{\displaystyle}
\newcommand{\beq}{\begin{eqnarray}}
\newcommand{\eeq}{\end{eqnarray}}
\newcommand{\beqs}{\begin{eqnarray*}}
\newcommand{\eeqs}{\end{eqnarray*}}
\newcommand{\vsp}{\vspace{0.1in}}
\newcommand{\nnu}{\nonumber}

\newcommand{\goto}{\rightarrow}

\newcommand{\by}{\mathbf{y}}

\newcommand{\bx}{\mbox{\boldmath $x$}}

\newcommand{\ep}{\epsilon}

\newcommand{\field}[1]{\mathbb{#1}} 
\newtheorem{theorem}{Theorem}[section]

\begin{document}

\markboth{Miao-jung Y. Ou}{Nonstandard Pad\'{e} approximants for effective properties}
\title{On nonstandard Pad\'{e} approximants suitable for effective properties of two-phase composite materials}

\author{Miao-jung Y. Ou\\UT Joint Institute for Computational Sciences, Oak Ridge National Laboratory,\\
Oak Ridge, TN 37831, USA\footnote{Building 6012, MS 6367, 1 Bethel Valley Road, Oak Ridge, TN 37831, USA}\\
mou@utk.edu}

\maketitle
\begin{abstract}
This paper investigates existence of the nonstandard Pad\'{e} approximants introduced by Cherkaev and Zhang in~\cite{Zh-Ch-09} for approximating the spectral function of composites from effective properties at different frequencies. The spectral functions contain microstructure information.  Since this reconstruction problem is ill-posed ~\cite{Elena-01}, the well-performed Pad\'{e} approach is noteworthy and deserves further investigations. In this paper, we validate the assumption that the effective dielectric component of interest of all binary composites can be approximated by Pad\'{e} approximants whose denominator has nonzero power one term. We refer to this as the nonstandard Pad\'{e} approximant, in contrast to the standard approximants whose denominator have nonzero constant terms. For composites whose spectral function assumes infinitely many different values such as the checkerboard microstructure, the proof is carried by using classic results for Stieltjes functions. For those with spectral functions having only finitely many different values, we prove the results by utilizing a special product decomposition of the coefficient matrix of the Pad\'{e} system. The results in this paper can be considered as an extension of the Pad\'{e} theory for Stieltjes functions whose spectral function take infinitely many different values to those taking only finitely many values. In the literature, the latter is usually excluded from the definition of Stieltjes functions because they correspond to rational functions, hence convergence of their Pad\'{e} approximants is trivial. However, from an inverse problem point of view, our main concern is the existence of the nonstandard Pad\'{e} approximants, rather than their convergence.  The results in this paper provide a  mathematical foundation for applying the Pad\'{e} approach for reconstructing the spectral functions of composites whose microstructure is not {\it a priori} known.
\end{abstract}


\section{Introduction}	
In the seminal paper by Golden and Papanicolaou\cite{GoPa-83}, components of the effective dielectric matrix of a two-pahse composite with statistically homogeneous microstructure and isotropic constituents are represented as  Stieltjes integrals with positive Borel measure supported in $[0,1]$. We refer to this integral representation formula as IRF throughout the paper. In the IRF, the contrast of the constituents, represented as the ratio (denoted by $h$) of the two dielectric constants, stays in the integrand while the information of microstructure is contained in the positive Borel measure. To be more specific, it is shown in \cite{GoPa-83} that for $s:=\frac{1}{1-h}$ outside $[0,1]$ on the complex plane, the $(i,k)$-component of the effective dielectric tensor, $i,k=1,2,3$ of a dielectric composite, denoted by $\ep_{ik}^*$ can be represented as 
\[
\ep_{ik}^*=\ep_1 \left(\delta_{ik} - \int_0^1 \frac{1}{s-z}d\mu_{ik}(z)\right)
\]  
where $h:=\frac{\ep_2}{\ep_1}$, $\delta_{ik}$ is Kronecker delta, and $d\mu_{ik}(z)$ is a positive Borel measure. Throughout the paper, we adapt the notation $:=$ and $=:$ such that
\beqs
A:=B \Leftrightarrow \mbox{ $A$ is defined by $B$}\\
A=:B \Leftrightarrow\mbox{ $B$ is defined by $A$} 
\eeqs

Following the convention in literature, we introduce $F(s)$ function 
\beq
F_{ik}(s)\DEF \delta_{ik} - \frac{\ep_{ik}^*}{\ep_1}=\int_0^1 \frac{1}{s-z}d\mu_{ik}(z)
\label{F_function}
\eeq
Note that $h$ is complex-valued for lossy materials. Applying geometric series expansion of $\frac{1}{s-z}$ around $s=\infty$ (corresponding to a homogeneous material), (\ref{F_function}) becomes
\[
F_{ik}(s)=\sum_{m=0}^\infty \frac{c_m}{s^{m+1}} 
\label{moment_exp}
\]
with $c_m$ being the $m$-th moment of the measure, i.e. $c_m\DEF \int_0^1 z^m \mu_{ik}(dz)$.

It is also shown in \cite{GoPa-83} that $c_m$, $m=0,1,2,\cdots$, is related to the $(m+1)$-point correlation function of the characteristic function of the region occupied by one of the constituents. This link has been exploited for estimating the bounds on effective dielectric tensor by using information on $m$-point correlation functions such as volume fraction ($m=1$) and two-point correlation function. This approach of estimating bounds relies on the observation from (\ref{F_function}) that the $F_{ik}(s)$ is an analytic function for $s\in\field{C}\setminus [0,1]$, where $\field{C}$ is the complex plane.  This provides a mathematical scheme for retrieving the moments $c_m$ when the dielectric constant of the one of the constituents varies with frequency. In \cite{Cher-Ou-08}, it was demonstrated analytically and numerically that the construction of moments of the measure from data of $\ep_{ik}^*$ at various frequencies is a stable process, unlike the ill-posed process of reconstructing the measure itself.  In that paper, the reconstruction of moments was carried out by solving a truncated linear system if all the data points $s_k$ are outside the unit circle and by a Gaussian quadrature method. In all the tested numerical examples, the first 5-10 moments can be reconstructed with high accuracy by both methods. These moments contain the microstructural information and the resulting truncated series also provide a very accurate way for calculating the effective dielectric constant for all frequencies(extrapolation) \cite{Cher-Ou-08}. In \cite{Zh-Ch-09}, instead of reconstructing the coefficients of the truncated Stieltjes series, the coefficients of the nonstandard Pad\'{e} approximants of the series in (\ref{moment_exp}) was reconstructed from data of $\ep^*_{ik}$ at different frequencies. This method is able to reconstruct the first 10-15 moments with high accuracy and the resulting Pad\'{e} appriximant  provides a very accurate way for calculating the effective dielectric constant for all frequencies(extrapolation). It is worth noticing that all the real-valued poles of the resulting Pad\'{e} approximant lie along the graph of the true spectral functions in all the tested numerical examples. To be more specific, in \cite{Zh-Ch-09}, Pad\'{e} approximants for the Stieltjes integral in (\ref{F_function}) with the following form:
\beq
F(s)  \simeq\frac{a(s)}{b(s)}=\frac{a_0 +a_1s +a_2s^2 + \cdot\cdot\cdot +a_ps^p}
         {b_0 +b_1 s +b_2s^2 + \cdot\cdot\cdot + b_qs^q},\, b_1\ne 0.
\label{F_pade}
\eeq 
was used for recovering the moments. They assume $F(s)$ has at least one pole and normalize $b_1$ to be 1, unlike the standard Pad\'{e} approximants for which $b_0$ is normalized to 1. We refer to (\ref{F_pade}) as \textbf{nonstandard Pad\'{e} apprximants}. The main reason prompting the choice of $b_1\ne 0$ is beecause for a checkerboard microstructure, the spectral function has a pole at $s=0$. For ease of notation, we will replace $\ep^*_{ik}$ with $\ep^*$ hereafter. $\mu$ is assumed to take at least two different values because otherwise  $F$ will be identically zero as implied by (\ref{F_function}). For $N$ different frequencies $\omega_1,\cdots, \omega_N$, there are $N$ data pairs $(s_k, d_k)$, where $s_k=\frac{1}{1-\ep_2(\omega_k)/\ep_1(\omega_k)}$ and $d_k=\ep^*(s_k)$, $k=1,\cdots,N$. The following  linear system of equations for $a_0, \cdots, a_p, b_0, b_2, \cdots, b_q$, 
\beq
d_k=\frac{a_0 +a_1s_k +a_2s_k^2 + \cdot\cdot\cdot +a_ps_k^p}
         {b_0 + s_k +b_2s_k^2 + \cdot\cdot\cdot + b_qs_k^q}, p+q+1\le N
\eeq
was solved by formulating it as a least-square minimization problem subject to the constraints
\beq
0\le A_n<1,\, 0\le s_n<1,\, 0<\sum_n A_n<1
\label{constraints}
\eeq
for the residues $A_n$ and poles $s_n$ defined by the partial fraction decomposition of the nonstandard Pad\'{e} approximant
\[
\frac{a_0 +a_1s +a_2s^2 + \cdot\cdot\cdot +a_ps^p}
         {b_0 +b_1 s +b_2s^2 + \cdot\cdot\cdot + b_qs^q}=: \sum_{n=1}^q \frac{A_n}{s-s_n}
\]
 The moments were then computed  by applying partial fraction decomposition to the Pad\'{e} approximants and by correctly combining the poles and residues.  See \cite{Zh-Ch-09,Ch-Zh-09} for details.
\vsp

In~\cite{En-05}\cite{TT-97}\cite{TT-98}, the fact that a Stieltjes function can be bounded up and below by standard Pad\'{e} approximants was used for deriving bounds on the effective dielectric constants.
\vsp

The main purpose of this paper is twofold. The first is to show that $F(s)$ can always be approximated by a Pad\'{e} approximant of the form in (\ref{F_pade}) for $s\in \field{C}\setminus[0,1]$, i.e. to show that we can always normalize $b_1$ to be 1 and to show this special form of Pad\'{e} approximant has the accuracy-through-order property that will be defined in the next section.  The second is to justify the constraints shown in (\ref{constraints}).
\vsp

This paper is organized as follows. We first define an auxiliary function $f(\xi)$ and prove that it has accuracy-through-order standard Pad\'{e} approximants in Section \ref{pade_f}. In the same section, we also prove properties of the poles of the Pad\'{e} approximants. Using these results, we show in Section \ref{main} that $F(s)$ always has accurate-through-order Pad\'{e} approximant of the form as in (\ref{F_pade}) for appropriate $p$ and $q$. The major technique in the proof is fractorizing the coefficient matrix of the system of Pad\'{e} equations into products of Vandermonde matrices and a diagonal matrix. In Section \ref{conclusion}, we discuss the relevance of this work in materials science for binary composites.
\section{Existence of Pad\'{e} approximants}
Since the representation formula for $F(s)$ is a Stieltjes integral with non-decreasing distribution function $\mu$ and is analytic for $s \in\field{C}\setminus [0,1]$, we may identify it with a Stieltjes function $f(\xi)$ \cite{Baker} by rewriting $F(s)$ in terms of a new variable $\xi:=-1/s$  
\beq
F(s)=:G(\xi)=-\xi\int_0^1 \frac{d\mu(z)}{1+z\xi}=:-\xi\cdot f(\xi)
\label{def_f_xi}
\eeq
We refer to $f(\xi)$ as the \textbf{auxiliary function} and would like to remark that the function $\mu(z)$ can take finitely many or infinitely many different values, i.e. the definition of Stieltjes function we use here does not include the restriction on the number of different values $\mu(z)$ takes. Note that the Stieltjes function $f(\xi)$ is analytic for $\xi\in\field{C}\setminus (-\infty,-1]$ and has a power series expansion valid in $|\xi|<1$, which reads
\beq
f(\xi)=\DIS\sum_{n=0}^{\infty} \mu_n (-\xi)^n. \label{power_expan_f}
\eeq 
For later use, we also include the expansion of $G$ in the same region on the $\xi$-plane
\beq
G(\xi)=\DIS\sum_{n=0}^{\infty} \mu_n (-\xi)^{n+1}. \label{power_expan_G}
\eeq
Here we use $-\xi$ because it is easier to work with the positive coefficients $\{\mu_n\}$, which can be easily verified as  the moments of the Borel measure $\mu$ in (\ref{def_f_xi}), $n=0,1,2,\cdots$. 
\subsection{\label{pade_f}Existence of standard Pad\'{e} approximants for the auxiliary function $f(\xi)$}
Recall the standard $[L/M]$ Pad\'{e} approximant for $f(\xi)$ in (\ref{power_expan_f}) is a rational function of the following form
\[
[L/M](-\xi):=\DIS\frac{a_0+a_1(-\xi)+\cdots+a_L(-\xi)^L}{b_0
+b_1(-\xi)+\cdots+b_M(-\xi)^M}
\]  
such that $b_0$ can be normalized to be 1 and the rest of the coefficients satisfy the system of Pad\'{e} equations \cite{Baker} 
\beq
&&\left[
\begin{array}{cccc}
\mu_{L-M+1} & \mu_{L-M+2} & \cdots & \mu_L\\
\mu_{L-M+2} & \mu_{L-M+3} & \cdots & \mu_{L+1}\\
\vdots & \vdots & \ddots &\vdots\\
\mu_{L} & \mu_{L+1} & \cdots & \mu_{L+M-1}
\end{array}
\right]
\left[
\begin{array}{c}
b_M\\b_{M-1}\\ \vdots \\ b_1
\end{array}
\right]
=-\left[
\begin{array}{c}
\mu_{L+1}\\\mu_{L+2}\\ \vdots \\ \mu_{L+M}
\end{array}
\right]
\label{Pade_eqns_b}
\eeq
\beq
\left\{
\begin{array}{l}
a_0=\mu_0\\
a_1=\mu_1+b_1 \mu_0\\
a_2=\mu_2+b_1\mu_1+b_2\mu_0\\
\vdots\\
a_L=\mu_L+\DIS\sum_{k=1}^{\min(L,M)} b_k \mu_{L-k}
\end{array}
\right.
\label{Pade_eqns_a}
\eeq 
For consistency, we define $\mu_j=0$ for $j<0$. The $[L/M]$ Pad\'{e} approximant is said to have the accuracy-through-order property if it satisfies 
\beq
f(\xi)-[L/M]=O(\xi^{L+M+1}).
\label{a-t-o-property}
\eeq
for $|\xi|$ within radius of convergence.
We would like to remark that for any given function $f(z)$, polynomials $p^{[L/M]}(z)$, $q^{[L/M]}(z)$ of degrees $L$, $M$, respectively, can always be found so that 
\beq
q^{[L/M]}(z)f(z)-p^{[L/M]}(z)=O(z^{L+M+1})
\label{p-q}
\eeq
However, it is well known that (\ref{p-q}) does not necessarily imply $p^{[L/M]}(z)/q^{[L/M]}(z)=f(z)+O(z^{L+M+1})$ \cite{Baker}. Following Baker, we say that the Pad\'{e} approximant does not exist if the accuracy-through-order requirement is not satisfied. On the other hand, it can be shown that \cite{Baker}(p.7) the rational function $\DIS\frac{P^{[L/M]}(-\xi)}{Q^{[L/M]}(-\xi)}$
with $P^{[L/M]}$ and $Q^{[L/M]}$ defined as
\beq
P^{[L/M]}(-\xi)=&&\left|
\begin{array}{cccc}
\mu_{L-M+1} & \mu_{L-M+2} & \cdots & \mu_{L+1}\\
\mu_{L-M+2} & \mu_{L-M+3} & \cdots & \mu_{L+2}\\
\vdots & \vdots & \ddots &\vdots\\
\mu_{L-1} & \mu_{L} & \cdots & \mu_{L+M-1}\\
\mu_{L} & \mu_{L+1} & \cdots & \mu_{L+M}\\
\DIS\sum_{i=0}^{L-M}\mu_i (-\xi)^{M+i} & \DIS\sum_{i=0}^{L-M+1}\mu_i (-\xi)^{M+i-1} & \cdots & \DIS\sum_{i=0}^{L}\mu_i (-\xi)^{i}
\end{array}
\right|
\label{Pade_P}
\eeq
\beq
Q^{[L/M]}(-\xi)=&&\left|
\begin{array}{ccccc}
\mu_{L-M+1} & \mu_{L-M+2} & \cdots & \mu_{L}&\mu_{L+1}\\
\mu_{L-M+2} & \mu_{L-M+3} & \cdots &\mu_{L+1} & \mu_{L+2}\\
\vdots & \vdots & \ddots &\vdots\\
\mu_{L-1} & \mu_{L} & \cdots & \mu_{L+M-2} & \mu_{L+M-1}\\
\mu_{L} & \mu_{L+1} & \cdots & \mu_{L+M-1} & \mu_{L+M}\\
(-\xi)^{M} & (-\xi)^{M-1} &\cdots& (-\xi)& 1
\end{array}
\right|
\label{Pade_Q}
\eeq
is an accuracy-through-order Pad\'{e} approximant of $f(\xi)$ provided $Q^{[L/M]}(0)\ne 0$. Note that $Q^{[L/M]}(0)$(Hankel determinant) is exactly the determinant of the coefficient matrix of (\ref{Pade_eqns_b}). Hence the condition $Q^{[L/M]}(0)\ne 0$ guarantees the accuracy-through-order property of the $[L/M]$ approximant \cite{Baker}(p.21). If $Q^{[L/M]}(0)=0$, (\ref{Pade_eqns_b}) can have infinitely many solutions or fail to have any solution, i.e. the $[L/M]$ Pad\'{e} approximant may not exist. We would like to remark that $Q^{[L/M]}(0)=0$ does not imply non-existence of accuracy-through-order Pad\'{e} approximants. For example, for a rational function with degree $p$ numerator and degree $q$ denominator, it is known that $Q^{[L/M]}(0)=0$ if $L> p$ and $M> q$ but obviously, the Pad\'{e} approximants are the rational function itself for these cases.
\vsp

Because the numerical algorithm in \cite{Zh-Ch-09} is based on the reconstruction of the Pad\'{e} approximant of $F(s)$, rather than $F(s)$ itself, it is necessary to prove the existence of Pad\'{e} approximant for $G(\xi)$ defined in (\ref{def_f_xi}). To achieve this, we first show the existence of standard Pad\'{e} approximants for the Stieltjes function $f(\xi)$ in the following theorem. 
\vsp

\begin{theorem}
\label{f-thm}
The $[L/M]$ Pad\'{e} approximant for $f(\xi)$ exists for all $L-M+1\ge 0$, $L\ne -1$ and $M\ge 0$ such that $b_0=1$.
\label{thm_pade_exist}
\end{theorem}
\begin{proof}
We first consider the case where the non-decreasing function $\mu(z)$ in (\ref{def_f_xi}) takes infinitely many different values. Using Theorem 5.2.1 in \cite{Baker}, which states that $Q^{[L/M]}(0)\ne 0$ for all $L-M+1\ge 0$ and $M\ge 1$, the existence of $[L/M]$ approximants follows directly. For $M=0$, $[L/0]$ clearly exists because $[L/0]=\sum_{j=0}^L \mu_j (-\xi)^j$.
\vsp

For the case in which $\mu(z)$ takes $n+1$ different values for some non-negative finite integer $n$, we prove the theorem as follows. Let $z_1,z_2,\cdots z_n$ be the $n$ different points on $[0,1]$ such that $0\le z_1 <z_2 <z_3<\cdots<z_n$ and $\mu(z)$ changes values only at these points, i.e. $d\mu(z)$ is a finite sum of Dirac measure sitting at $z_i$ with strength $\lambda_i$.
\vsp

\underline{\textbf{Case 1: $\mathbf{n=0}$} }
If $n=0$, then $f(\xi)\equiv 0$ so $[L/M]=0$ for all $L-M+1\ge 0$ and $M\ge 0$. 
\vsp

\underline{\textbf{Case 2: $\mathbf{n=1}$} } If $z_1=0$, i.e. $f(\xi)$ takes two different values and the value change occurs at $z_1$ by the amount of $\lambda_1$, then $f(\xi)\equiv\lambda_1>0$. Clearly, the Pad\'{e} approximant exits. If $z_1>0$, $f(\xi)=\frac{\lambda}{1+z_1\xi}$ so $[L/M]=\frac{\lambda}{1+z_1\xi}$ for all $L-M+1\ge 0$ and $M\ge 1$ by Theorem 1.4.4 in \cite{Baker} and $[L/0]$ clearly exists because $[L/0]=\sum_{j=0}^L \mu_j (-\xi)^j$.  
\vsp

\underline{\textbf{Case 3: $\mathbf{n\ge 2}$}}  Denote the amount of value change at $z_i$ by $\lambda_i>0$, $i=1,2,\cdots,n$.  With this, $f(\xi)$ can be written as
\beqs
f(\xi)=\DIS\sum_{i=1}^n \frac{\lambda_i}{1+z_i\xi}
\eeqs   
Suppose $z_1 \ne 0$, i.e.  $z_i, i=1\sim n$ are all non-zero. Then $f(\xi)$ is a rational function with degree $n-1$ numerator and degree $n$ denominator. By the characterization theorem of rational function (Theorem 1.4.4) in \cite{Baker}, we know that $[L/M](-\xi)=f(\xi)$ if $L\ge n-1$ and $M\ge n$, i.e. the Pad\'{e} approximant of a rational function is itself if both $L$ and $M$ are large enough. Also, the constant term of the denominator is obviously non-zero. This takes care of those $[L/M]$ such that $L-M+1\ge 0$ and $M\ge n$.
\vsp

To complete the proof for $z_1\ne 0$, we consider $[L/M]$ such that $L-M+1\ge 0$ and $1\le M < n$.  For this case, we exploit the the structure of moments 
$
\mu_0=\sum_{i=1}^n \lambda_i; \,\mu_k=\sum_{i=1}^n \lambda_i z_i^k,\,k=1,2,3\dots
$
to rewrite the coefficient matrix of (\ref{Pade_eqns_b}) as
\beq
\left[
\begin{array}{llll}
\DIS\sum_{i=1}^n \lambda_i z_i^{L-M+1} & \DIS\sum_{i=1}^n \lambda_i z_i^{L-M+2} & \cdots & \DIS\sum_{i=1}^n \lambda_i z_i^{L}\\
\DIS\sum_{i=1}^n \lambda_i z_i^{L-M+2}& \DIS\sum_{i=1}^n \lambda_i z_i^{L-M+3}& \cdots &\DIS\sum_{i=1}^n \lambda_i z_i^{L+1}\\
\vdots & \vdots & \ddots &\vdots\\
\DIS\sum_{i=1}^n \lambda_i z_i^{L} & \DIS\sum_{i=1}^n \lambda_i z_i^{L+1} & \cdots & \DIS\sum_{i=1}^n \lambda_i z_i^{L+M-1}
\end{array}
\right]
={\mathbf{V}}^T \mathbf{\Lambda} \mathbf{V},
\label{decomp_1}
\eeq   
where
\beq
\mathbf{V}:=\left[
\begin{array}{lllll}
1&z_1&z_1^2&\cdots&z_1^{M-1}\\
1&z_2&z_2^2&\cdots&z_2^{M-1}\\
\vdots&\vdots & \vdots & \ddots &\vdots\\
1&z_M&z_M^2&\cdots& z_M^{M-1}\\
\vdots&\vdots & \vdots & \ddots &\vdots\\
1&z_n&z_n^2&\cdots& z_n^{M-1}
\end{array}
\right]_{n\times M}\\
\mathbf{\Lambda}:=\left[
\begin{array}{lllll}
\lambda_1 z_1^{L-M+1}&0&0&\cdots&0\\
0& \lambda_2 z_2^{L-M+1}&0&\cdots&0\\
0&0& \lambda_3 z_3^{L-M+1}&\cdots&0\\
\vdots&\vdots & \vdots & \ddots &\vdots\\
0&0&0&\cdots& \lambda_n z_n^{L-M+1}
\end{array}
\right]_{n\times n}
\eeq
The superscript $T$ denotes matrix transpose. Note that the null space $\mathcal{N}({\mathbf{V}})=\mathbf{0}\in \field{R}^M$ because $\mathbf{V}$ is a rectangular Vandermonde matrix with $n$ rows and $M$ columns such that $n> M$. Let $\bx$ be a vector in $\field{R}^M$ so that ${\mathbf{V}}^T \mathbf{\Lambda} \mathbf{V}\bx=\mathbf{0}$. Consider the inner product $(\bx,\by):=\sum_{k=1}^n x_i y_i$ for $\bx=(x_1,\cdots, x_n)^T$,  $\by=(y_1,\cdots, y_n)^T$, then
\[
0=({\mathbf{V}}^T \mathbf{\Lambda} \mathbf{V}\bx,\bx)=(\mathbf{\Lambda} \mathbf{V}\bx,{\mathbf{V}}\bx) \ge  \min_{i=1\sim n} \lambda_i z_i^{L-M+1} ( {\mathbf{V}}\bx,{\mathbf{V}}\bx)
\] 
Since $\D\min_{i=1\sim n} \lambda_i z_i^{L-M+1} >0$, we have ${\mathbf{V}}\bx=\mathbf{0}$. Hence $\bx=\mathbf{0}\in \field{R}^M$. This shows that the coefficient matrix in (\ref{Pade_eqns_b}) is non-singular so the $[L/M]$ approximant exists for $L-M+1\ge 0$ and $1\le M< n$ if $z_i\ne 0$ for every $i$.  
\vsp
\begin{table}
\caption{\label{large_L_M}Schematic description for case $n\ge 2$ and $z_1=0$}
\begin{center}
\begin{tabular}[ht]{l|cccccccccc}
$M\setminus L$ &0 &1 &2 &$\cdots$ &$n-3$ &$n-2$&$n-1$&$n$&$n+1$ &$\cdots$\\ \hline 
0 &$\bullet$&$\bullet$&$\bullet$&$\cdots$&$\bullet$&$\bullet$&$\bullet$&$\bullet$&$\bullet$&$\cdots$\\
1 &$\circ$&$\bullet$&$\bullet$&$\cdots$&$\bullet$&$\bullet$&$\bullet$&$\bullet$&$\bullet$&$\cdots$\\
2 & &$\circ$&$\bullet$&$\cdots$&$\bullet$&$\bullet$&$\bullet$&$\bullet$&$\bullet$&$\cdots$\\
3 & & &$\circ$&$\cdots$&$\bullet$&$\bullet$&$\bullet$&$\bullet$&$\bullet$&$\cdots$\\
$\vdots$ & & & & $\ddots$ &$\vdots$&$\vdots$&$\vdots$&$\vdots$&$\vdots$&$\ddots$\\
$n-2$ & & & &  &$\circ$&$\bullet$&$\bullet$&$\bullet$&$\bullet$&$\cdots$\\
$n-1$ & & & &  & &$\circ$&$\times$&$\times$&$\times$&$\cdots$\\
$n$     & & & &  & & &$\circ$&$\times$&$\times$&$\cdots$\\
$n+1$     & & & &  & & & &$\circ$&$\times$&$\cdots$\\
$\vdots$  & & & &  & & & & &$\ddots$ &$\ddots$
\end{tabular}
\end{center}
\end{table}
%
\vsp

If $z_1=0$, $f(\xi)$ is a rational function with degree $n-1$ numerator and degree $n-1$ denominator. The characterization theorem of rational function (Theorem 1.4.4) in \cite{Baker} implies that $[L/M](-\xi)=f(\xi)$ if $L\ge n-1$ and $M\ge n-1$. This takes care of $L-M+1\ge 1$ and $M\ge n-1$, represented by  the triangular region marked by $\times$ in Table \ref{large_L_M}. 
For $L-M+1=0$ and $M\le n-1$, the coefficient matrix in (\ref{Pade_eqns_b}) can be decomposed in the same fashion as in (\ref{decomp_1}) with $\mathbf{\Lambda}$ replaced by
\beq
\mathbf{\Lambda}:=\left[
\begin{array}{lllll}
\lambda_1 &0&0&\cdots&0\\
0& \lambda_2 &0&\cdots&0\\
0&0& \lambda_3 &\cdots&0\\
\vdots&\vdots & \vdots & \ddots &\vdots\\
0&0&0&\cdots& \lambda_n 
\end{array}
\right]_{n\times n}
\label{special_Lambda}
\eeq   
By an argument exactly the same as for the $z_1\ne 0$ case, the existence of the $[L/M]$ Pad\'{e} approximant for $L-M+1=0$ and $M\le n-1$ is proved; they correspond to the line formed by $\circ$ in Table \ref{large_L_M}. For $L-M+1>0$ and $M\le n-1$, the coefficient matrix of (\ref{Pade_eqns_b}) is
\beq
\left[
\begin{array}{llll}
\DIS\sum_{i=2}^n \lambda_i z_i^{L-M+1} & \DIS\sum_{i=2}^n \lambda_i z_i^{L-M+2} & \cdots & \DIS\sum_{i=2}^n \lambda_i z_i^{L}\\
\DIS\sum_{i=2}^n \lambda_i z_i^{L-M+2}& \DIS\sum_{i=2}^n \lambda_i z_i^{L-M+3}& \cdots &\DIS\sum_{i=2}^n \lambda_i z_i^{L+1}\\
\vdots & \vdots & \ddots &\vdots\\
\DIS\sum_{i=2}^n \lambda_i z_i^{L} & \DIS\sum_{i=2}^n \lambda_i z_i^{L+1} & \cdots & \DIS\sum_{i=2}^n \lambda_i z_i^{L+M-1}
\end{array}
\right]
={\mathbf{V}}^T \mathbf{\Lambda} \mathbf{V},
\label{decomp_2}
\eeq   
where
\beq
\mathbf{V}:=\left[
\begin{array}{lllll}
1&z_2&z_2^2&\cdots&z_2^{M-1}\\
1&z_3&z_3^2&\cdots&z_3^{M-1}\\
\vdots&\vdots & \vdots & \ddots &\vdots\\
1&z_M&z_M^2&\cdots& z_M^{M-1}\\
\vdots&\vdots & \vdots & \ddots &\vdots\\
1&z_n&z_n^2&\cdots& z_n^{M-1}
\end{array}
\right]_{(n-1)\times M}\label{special_V}\\
\mathbf{\Lambda}:=\left[
\begin{array}{lllll}
\lambda_2 z_2^{L-M+1}&0&0&\cdots&0\\
0& \lambda_3 z_3^{L-M+1}&0&\cdots&0\\
0&0& \lambda_4 z_4^{L-M+1}&\cdots&0\\
\vdots&\vdots & \vdots & \ddots &\vdots\\
0&0&0&\cdots& \lambda_n z_n^{L-M+1}
\end{array}
\right]_{(n-1)\times (n-1)}
\label{special_Lambda2}
\eeq
Since $\mathbf{V}$ is a Vandermonde matrix with $n-1$ rows and $M$ columns such that $n-1 \ge M$ and $z_i\ne z_j$ for $i\ne j$, the null space $\mathcal{N}(\mathbf{V})=\mathbf{0}\in \field{R}^M$. Because all the entries in $\mathbf{\Lambda}$ are positive, the existence of the $[L/M]$ Pad\'{e} approximant can be concluded in the same way as for the previous case. The corresponding region in the Pad\'{e} table is marked in Table \ref{large_L_M} by $\bullet$.
\end{proof}
We apply this result to show the existence of accuracy-through-order Pad\'{e} approximants for $G(\xi)$.

\subsection{Standard Pad\'{e} approximants for $G(\xi)$}
\begin{theorem}
The $[L/M]$ Pad\'{e} approximant for $G(\xi)$ exists for all $L-M\ge 0$ and $M\ge 0$ such that $b_0=1$.
\label{main_thm}
\end{theorem}
\begin{proof}
For $L\ge 1$, let $[L-1/M]_f$ denote the $[L-1/M]$ Pad\'{e} approximant of $f(\xi)$ such that $(L-1)-M+1\ge 0$, $L\ge 1$ and $M\ge 0$. The proof follows immediately from the definition of $G$, Theorem \ref{f-thm} and definition of accuracy-through-order property (\ref{a-t-o-property}):
\[
G(\xi)-(-\xi)[L-1/M]_f=-\xi\left(f(\xi)-[L-1/M]_f\right)=O(\xi^{L+M+1}) \mbox{ for } |\xi|<1. 
\]
That is, $(-\xi)[L-1/M]_f$ is the $[L/M]$ Pad\'{e} approximant of $G(\xi)$ for $L\ge 1$ and $L\ge M\ge 0$. If $L=0$, then $M=0$ because of the condition $L\ge M$. The $[0/0]$ approximant of $G$ is $0$ because $G(\xi)-0=\mu_0 (-\xi)+\mu_1 (-\xi)^2+\cdots=O(\xi^{0+0+1})$. The last equality is due to the positivity of $\mu_0$.
\end{proof}
 The algorithm presented in \cite{Zh-Ch-09} for reconstructing spectral function from known measured data $\ep^*(\xi)$ for $\xi$ not necessarily in $|\xi|<1$. This can be justified by the two theorems listed below regarding properties of Pad\'{e} approximants of Stieltjes series. 
\begin{theorem}
The poles of the Pad\'{e} approximants of $f$, denoted by $[L/M]_f$, $L-M+1\ge 0$ and $M\ge 0$ are simple poles lying on $(-\infty,-1)$ with positive residues.
\label{thm_simple_pole}
\end{theorem}
\begin{proof}
For the case where the measure $\mu$ takes infinitely many different values, the proof can be found in \cite{Baker}, Theorem 5.2.1, p201. When $\mu$ takes $(n+1)<\infty$ different values, the proof given in \cite{Baker} needs to be modified as follows. The key step is to show the interlacing property of the zeros of $\triangle^{(J)}_M(x)$ defined in (\ref{delta}).
\vsp

If $M=0$, the Pad\'{e} approximant has no pole.  For $n=0,1$, the proof is the same as that for Theorem \ref{f-thm}. 
\vsp

For $n\ge 2$, we first consider the case when $z_1\ne 0$, i.e. no jump in values at $z=0$ hence $f(\xi)$ is a rational function with numerator of degree $n-1$ and denominator of degree $n$. Let $J:=L-M\ge -1$, and $M\ge 1$. Following the proof of Theorem 5.2.1 in \cite{Baker}, we define $\triangle^{(J)}_0(x)=1$; for $M=1,2,3 \cdots$, define
\beq
&&\triangle^{(J)}_M(x):=\nnu\\
&&\left|
\begin{array}{cccc}
\mu_{1+J}+x \mu_{2+J} & \mu_{2+J}+ x \mu_{3+J} & \cdots & \mu_{M+J}+x \mu_{M+J+1}\\
\mu_{2+J}+x \mu_{3+J} & \mu_{3+J}+x \mu_{4+J}& \cdots & \mu_{M+J+1}+x \mu_{M+J+2}\\
\vdots & \vdots & \ddots &\vdots\\
\mu_{M+J}+x \mu_{M+J+1} & \mu_{M+J+1}+x \mu_{M+J+2}& \cdots & \mu_{2M+J-1}+x \mu_{2M+J}\end{array}
\right|
\label{delta}
\eeq
Note that $\triangle^{(J)}_M(x)=Q^{[L/M]}(-x)$ by applying (3.2) on p.16 of \cite{Baker} to (\ref{Pade_Q}) and that $\triangle_M^{(J)}(0)=0$ for $M>n$ and $J\ge -1$. Hence the zeros of $\triangle^{(J)}_M$ correspond to the poles of Pad\'{e} approximant $[L/M]$. Unlike the proof for $n=\infty$ in \cite{Baker}, Theorem 5.2.1, we only need to consider $J\ge -1$ and $1\le M< n$ because for $J\ge -1$ and $M\ge n$, the Pad\'{e} approximant for $f(\xi)$ is itself the poles of $f(\xi)$ with finite $n$ are apparently all simple and located in $(-\infty,-1)$. Next, we need to show that the functions $\triangle_1^{(J)}(x), \cdots, \triangle_{n-1}^{(J)}(x)$ have the property such that if $\triangle_M^{(J)}(x)=0$ for some $x$, then $\triangle_{M-1}^{(J)}(x)$ and $\triangle_{M+1}^{(J)}(x)$ have opposite signs for all $1\le M \le n-1$. This can be shown by using Sylvester's identity and Frobenius identity in exactly the same way as in the proof in  \cite{Baker}, Theorem 5.2.1. In order to show the interlacing property of the zeros of $\triangle_M^{(J)}(x)$,  we need to prove the following properties 
\beq
\left.
\begin{array}{rcl}
\triangle_M^{(J)}(0)&>&0\\
\triangle_M^{(J)}(-\infty)&=&(-1)^M\cdot\infty
\end{array}
\right\} \mbox{ for $M=1,\cdots, n$}
\label{interlacing}
\eeq
by first noting that $\triangle_M^{(J)}(0)$ is the determinant of the matrix ${\mathbf{V}}^T \mathbf{\Lambda} \mathbf{V}$ in (\ref{decomp_1}) and 
\[
({\mathbf{V}}^T \mathbf{\Lambda} \mathbf{V}\by,\by)=( \mathbf{\Lambda} \mathbf{V}\by,\mathbf{V}\by )\ge \min_{j=1,\cdots,n} \lambda_j z_j^{L-M+1}(\mathbf{V}\by,\mathbf{V}\by ) >0, \mbox{ for } \by \in \field{R}^M\setminus \{\mathbf{0}\}
\]
So $\det({\mathbf{V}}^T \mathbf{\Lambda} \mathbf{V})>0$, hence $\triangle_M^{(J)}(0)>0$ for $M=0,1,\cdots, n$. For $\triangle_M^{(J)}(-\infty)$, we decompose the matrix corresponding to $\triangle_M^{(J)}(x)$ as 
\[
\left[
\begin{array}{cccc}
\mu_{1+J}+x \mu_{2+J} & \mu_{2+J}+ x \mu_{3+J} & \cdots & \mu_{M+J}+x \mu_{M+J+1}\\
\mu_{2+J}+x \mu_{3+J} & \mu_{3+J}+x \mu_{4+J}& \cdots & \mu_{M+J+1}+x \mu_{M+J+2}\\
\vdots & \vdots & \ddots &\vdots\\
\mu_{M+J}+x \mu_{M+J+1} & \mu_{M+J+1}+x \mu_{M+J+2}& \cdots & \mu_{2M+J-1}+x \mu_{2M+J}\end{array}
\right]
={\mathbf{V}}^T \mathbf{\Lambda_x} \mathbf{V}
\]
with 
\[
\mathbf{\Lambda}_x:=\left[
\begin{array}{llll}
\lambda_1 z_1^{L-M+1}(1+z_1 x)&0&\cdots&0\\
0& \lambda_2 z_2^{L-M+1}(1+z_2 x)&\cdots&0\\
\vdots&\vdots &  \ddots &\vdots\\
0&0&\cdots& \lambda_n z_n^{L-M+1}(1+z_n x)
\end{array}
\right]_{n\times n}
\]
Similar as before, for a fixed $x$, there exists $A$ and $B$ such that
\[
A (\mathbf{V}\by,\mathbf{V}\by )\ge ({\mathbf{V}}^T \mathbf{\Lambda}_x \mathbf{V}\by,\by)=( \mathbf{\Lambda}_x \mathbf{V}\by,\mathbf{V}\by )\ge  B(\mathbf{V}\by,\mathbf{V}\by ) , \mbox{ for } \by \in \field{R}^M\setminus \{\mathbf{0}\}
\]
where
\beqs
A:=\max_{j=1,\cdots,n} \lambda_j z_j^{L-M+1}(1+z_j x),\,\,
B:=\min_{j=1,\cdots,n} \lambda_j z_j^{L-M+1}(1+z_j x)
\eeqs
As $x\goto -\infty$, we have $A\goto -\infty$ and $B\goto -\infty$. Hence $\triangle_M^{(J)}(x)$=$\det(\mathbf{V}^T \mathbf{\Lambda}_x \mathbf{V}) \goto (-\infty)^M$ as $x\goto -\infty$. This completes the proof of (\ref{interlacing}) for $z_1\ne 0$. 
\vsp

For $z_1=0$, we first consider that case when $L-M+1=0$ (i.e. $J=-1$) and $n-1\ge M\ge 1$. We have $\triangle_M^{(J)}(0)=\det({\mathbf{V}}^T \mathbf{\Lambda} \mathbf{V})$ with $\mathbf{\Lambda}$ defined in (\ref{special_Lambda}).  So $\triangle_M^{(J)}(0)>0$ for the same reason as in the $z_1\ne 0$ case. As for $\triangle_M^{(J)}(x)$ for fixed $x$, we have $\triangle_1^{(J)}(x)=\lambda_1+\D\sum^{n}_{i=2}\lambda_i (1+xz_i)\goto -\infty$ as $x\goto -\infty$ and for $M\ge 2$
\beqs
\triangle_M^{(J)}(x)=\det\left( 
\left[
\begin{array}{llll}
\lambda_1&0&\cdots&0\\
\vdots&\vdots &  \ddots &\vdots\\
0&0&\cdots& 0
\end{array}
\right]_{M\times M}
+\mathbf{V}^T \mathbf{\Lambda}_x \mathbf{V}
\right)
\eeqs
with $\mathbf{V}$ being the one defined in (\ref{special_V}) and 
\beq
\mathbf{\Lambda}_x:=\left[
\begin{array}{llll}
\lambda_2 (1+z_2 x)&0&\cdots&0\\
0& \lambda_3 (1+z_3 x)&\cdots&0\\
\vdots&\vdots &  \ddots &\vdots\\
0&0&\cdots& \lambda_n (1+z_n x)
\end{array}
\right]_{(n-1)\times (n-1)}
\label{special_Vx}
\eeq
Therefore, for any $\by =(y_1, \cdots, y_M)^T\ne \mathbf{0}$, we have
\beqs
\lambda_1 y_1^2+A(\mathbf{V}\by,\mathbf{V}\by)\ge
\left(\left( 
\left[
\begin{array}{llll}
\lambda_1&0&\cdots&0\\
\vdots&\vdots &  \ddots &\vdots\\
0&0&\cdots& 0
\end{array}
\right]_{M\times M}
+\mathbf{V}^T \mathbf{\Lambda}_x \mathbf{V}\right)
\by,\by\right)
\ge \lambda_1 y_1^2+B(\mathbf{V}\by,\mathbf{V}\by)
\eeqs
with $A:=\D\max_{i=2,\cdots,n} \lambda_i(1+z_i x)$ and $B:=\D\min_{i=2,\cdots,n} \lambda_i(1+z_i x)$; hence $\triangle_M^{(J)}(x)\goto (-\infty)^M$ as $x\goto -\infty$. For $L-M+1\ge 1$, $\triangle_M^{(J)}(0)=\det({\mathbf{V}}^T \mathbf{\Lambda} \mathbf{V})$ and $\triangle_M^{(J)}(x)=\det({\mathbf{V}}^T \mathbf{\Lambda}_x \mathbf{V})$ with $\mathbf{V}$ defined in (\ref{special_V}), $\mathbf{\Lambda}$ defined in (\ref{special_Lambda2}) and $\mathbf{\Lambda}_x$ defined in (\ref{special_Vx}) and the same argument can be applied for proving (\ref{interlacing}). The rest of the proof is the same as that in \cite{Baker}, Theorem 5.2.1.
\end{proof}
This theorem implies the following theorem:
%
\begin{theorem}[\cite{Baker}, Theorem 5.4.4, p225]
Any sequence of $[L_k/M_k]$ Pad\'{e} approximants, $L_k-M_k+1\ge 0$, of a Stieltjes series $f(\xi)$ convergent in $|\xi|<1$ converges uniformly to $f(\xi)$ in $\mathcal{D}^+(\triangle)$, where $\mathcal{D}^+(\triangle)$ is the set of all points in $|\xi|<R$ which are at least $\triangle$ distance away from the cut $(-\infty,-1]$ for any positive numbers $R$ and $\triangle$.
\label{thm_uniform_conv}  
\end{theorem}
It other words, even though the Pad\'{e} approximants are constructed from the Stieltjes series convergent only in $|\xi|<1$, their validity of approximating the Stieltjes function $f(\xi)$, which is analytic everywhere outside $(-\infty, -1]$, extends far beyond $|\xi|<1$.
\vsp

In \cite{Ch-Zh-09}, Cherkaev et. al. define $S_N$ equivalence of dielectric composites with isotropic constituents. According to their definition, two microstructures are $S_N$-equivalent if the the first $N$ moments ($\mu_0$ to $\mu_{n-1}$) of the spectral measures $\mu_{ik}^{(1)}$ and $\mu_{ik}^{(2)}$ in the IRF (\ref{F_function}) are identical. In the following, we show that Pad\'{e} approximants provide a characterization of $S_N$-equivalent structures in the following sense:
\begin{theorem} {\rm Let $G_1(\xi)$ and $G_2(\xi)$ be two Stieltjes functions defined as in (\ref{def_f_xi}) by the effective complex dielectric constant of Structure-1 and Structure-2, respectively. Denote the corresponding $[L/M]$ Pad\'{e} approximant by $[L/M]_{G_1}$ and $[L/M]_{G_2}$. Then Structure-1 and Structure-2 are $S_N$ equivalent, $N\ge 1$, if and only if $[L/M]_{G_1}=[L/M]_{G_2}$ for all $L+M\le N$ such that $L\ge M\ge 0$.}
\end{theorem}
\begin{proof}
Since the coefficients of $(-\xi)^{n+1}$, $n\ge 0$, in the power series expansions of $G_1$ and $G_2$ (\ref{power_expan_G}) are given by the $n^{th}$ moment of the corresponding measure, $S_N$-equivalence implies that the power series expansions near $\xi=0$ of $G_1$ and $G_2$ have the same first $N+1$ terms (the constant term in the series is $0$ for both $G_1$ and $G_2$ ). Since the $[L/M]$ Pad\'{e} approximant is completely determined by coefficients up to $(-\xi)^{L+M}$, clearly $[L/M]_{G_1}=[L/M]_{G_2}$ if $L+M\le N$ and if they both exist. By Theorem \ref{main_thm}, we conclude that $[L/M]_{G_1}=[L/M]_{G_2}$ for $L+M\le N$ and $L\ge M\ge 0$.
\vsp

Conversely, suppose $[L/M]_{G_1}(-\xi)=[L/M]_{G_2}(-\xi)$ for $L+M=N$. Due to Theorem \ref{thm_simple_pole} all the poles of Pad\'{e} approximants for Stieltjes series convergent in $|\xi|<1$ lie on the real axis in the interval $ -\infty<\xi<-1$, the power series expansion of this rational function 
\beqs
[L/M]_{G_1}(-\xi)=[L/M]_{G_2}(-\xi)=\DIS\sum_{n=0}^{\infty} c_n^{[L/M]}(-\xi)^n.
\eeqs
is valid for $|\xi|< 1$. By accuracy-through-order property of the Pad\'{e} approximants, $c_n^{[L/M]}=\mu_{n-1}^{(1)}$ and $c_n^{[L/M]}=\mu_{n-1}^{(2)}$ for $n=1\sim N$ with $\mu_n^{(1)}$ and $\mu_n^{(2)}$ being the coefficients of $G_1$ and $G_2$, respectively, as defined in (\ref{power_expan_G}). Hence the theorem is proved.
\end{proof}
%
\section{\label{main}Existence of nonstandard Pad\'{e} approximants for $F(s)$}
The main result of this paper is to show the existence of nonstandard Pad\'{e} approximants in (\ref{F_pade}).
\begin{theorem}
\label{main_thm}
For any Stieltjes function
\beq
F(s)=\int_0^1 \frac{1}{s-z} d\mu(z)
\label{main_F}
\eeq
with non-decreasing spectral function $\mu(z)$ taking either infinitely or finitely many different values for $z\in[0,1]$, there exits nonstandard Pad\'{e} approximant of the form
\beq
F(s)\simeq\frac{a(s)}{b(s)}=\frac{a_0 +a_1s +a_2s^2 + \cdot\cdot\cdot +a_{M-1}s^{M-1}}
         {b_0 +b_1 s +b_2s^2 + \cdot\cdot\cdot + b_Ms^M},\, b_1\ne 0, M\ge 1.
\label{main_Pade}
\eeq  
\end{theorem}
\begin{proof}
We first consider the case in which $\mu(z)$ takes infinitely many different values on $[0,1]$. Let $L= M-1$ and $M\ge 1$. According to Theorem \ref{thm_pade_exist} and Theorem \ref{thm_simple_pole}, the $[M-1/M]$ Pad\'{e} approximant of $f(\xi)$ exists with simple poles lying on $\xi\in(-\infty,-1)$ with positive residues. The interlacing-of-zero property of $\triangle_M^{(J)}(x)$ guarantees the existence of $M$ different poles of $[M-1/M]$. By Theorem \ref{thm_simple_pole}, it can be further expressed in terms of the partial fraction decomposition
\beq
[M-1/M]_f=\DIS\frac{P(-\xi)}{Q(-\xi)}=\DIS\sum_{i=1}^M \frac{\lambda_i}{\xi-p_i}, \, p_i \in (-\infty,-1),\, p_i\ne p_j \nnu\\\mbox{ for } i\ne j,\,\lambda_i>0.
\label{f_M}
\eeq
This induces the $[M/M]$ approximant of $G(\xi)$:
\beq
[M/M]_G=-\xi\cdot[L/M]_f=\DIS\sum_{i=1}^M \frac{-\xi\cdot\lambda_i}{\xi-p_i}, \, p_i \in (-\infty,-1),\,\lambda_i>0.
\eeq
We rewrite the expression above in terms of variable $s:=\frac{-1}{\xi}\in[0,1)$
\beq
[M/M]_G=\DIS\sum_{i=1}^M \frac{-\xi\cdot\lambda_i}{\xi-p_i}=\DIS\sum_{i=1}^M \frac{-\lambda_i/p_i}{s-(-1/p_i)}.
\label{pole_exp}
\eeq
Noting $\frac{-1}{p_i}\in[0,1)$ and $F(s)=G(\xi)$, the partial fraction representation above implies there exists a rational function approximation of $F(s)$ with the following form 
\[
F(s)\simeq\frac{a(s)}{b(s)}=\frac{a_0 +a_1s +a_2s^2 + \cdot\cdot\cdot +a_{M-1}s^{M-1}}
         {b_0 +b_1 s +b_2s^2 + \cdot\cdot\cdot + b_Ms^{M}},  M\ge 1
\]
where $b_1\ne 0$. This is because the coefficient of $s$ in the denominator of (\ref{pole_exp}) is $\sum_{i=1}^M \prod_{j\ne i} \frac{1}{p_j}$, which is non-zero for $M\ge 2$. For $M=1$, this is obviously true. 
\vsp

Next, we consider the case for $\mu$ taking $(n+1)$ different values, $1 \le n< \infty$ with jumps at $1>z_1,\dots, z_n \ge 0$ of magnitude $\lambda_1,\dots, \lambda_n>0$. Then
\beq
F(s)=\sum_{i=1}^n \frac{\lambda_i}{s-z_i}
\label{main_finite_F}
\eeq
For $M \ge n$, the $[M-1/M]$ Pad\'{e} approximant of $F(s)$ is itself and apparently (\ref{main_finite_F}) has the form (\ref{main_Pade}). For $M<n$, the proof is the same as in the case where $\mu$ takes infinitely many different values because of the validity of Theorem \ref{thm_pade_exist} and Theorem \ref{thm_simple_pole} for this case.
\end{proof}
\section{\label{conclusion}Conclusion and discussion}
In this paper, we prove that a Stieltjes function of the form (\ref{main_F}) can be approximated by $[M-1/M]$ non-standard Pad\'{e} approximant (\ref{main_Pade}). Most of the work in this paper is dedicated to the case where the spectral measure $\mu$ takes finitely many different values in $[0,1]$ because the proof for $\mu$ that takes infinitely many different values in $[0,1]$ can be obtained by results that are already available in the literature, as indicated in the theorems presented in this paper. Most results in the literature concern about the convergence of $[L/M]$ Pad\'{e} approximants to a Stieltjes function as $L$ and/or $M\goto \infty$. Naturally, they exclude those functions whose spectral function $\mu$ takes finitely many different values due to the fact that these functions are rational functions and $[L/M]$ Pad\'{e} approximants for a rational function is the function itself when $L$ and $M$ are large enough; hence convergence is guaranteed. The motivation for proving the main result presented in Theorem \ref{main_thm} is due to the concerns on its application on inverse homogenization as described in the introduction. It's well known that for any rank $n$ laminate microstructure,  $F(s)$ has finite number of poles, i.e. its spectral measure takes finite number of different values. At the same time, it is also known that the checkerboard microstructure has a spectral function which takes infinitely many different values and has a jump at $z=0$. In other words, given data on effective dielectric constant of a composite, we only know that the corresponding $F(s)$ function can be represented as a Stieltjes integral but we do not know {\it{a priori}}  how many different values the spectral function takes because the microstructure itself is the unknown for the inverse problem. From the inverse problem point of view, the existence of the nonstandard Pad\'{e} approximants is as important as the convergence property. As mentioned in the introduction, in \cite{Zh-Ch-09}, the spectral measure was reconstructed  with very good precision by using data of effective dielectric constants at different frequencies. It is done by recovering the coefficients $b_0, b_2,\dots, b_M,a_0,\dots, a_M$ by reformulating the problem as minimization problem with constraints on the location of poles $s_n$ and residues $A_n$ in (\ref{constraints}).  The results in this paper also justify these constraints except $\sum_n A_n<1$ by identifying $A_n$ with $-\frac{\lambda_n}{p_n}$ and $s_n$ with $-1/p_n$ in (\ref{pole_exp}). The constraint $\sum_n A_n<1$ is justified by noticing that $\sum_i \frac{-\lambda_i}{p_i}$ is the constant term of the Maclaurin series of $[M/M-1]_f$ in (\ref{f_M}) and it has to be equal to the constant term of the Maclaurin series of $f$, which is $\mu_0$ in (\ref{power_expan_f}). $\mu_0<1$ because it is the volume fraction of one of the constituent materials.
\vsp

{\bf Acknowledgments}  The author would like to thank E. Cherkaev and D. Zhang for bringing  into her attention the problem regarding existence of nonstandard Pad\'{e} approximants for IRF. This research is partially sponsored by the ARRA-NSF Grant DMS-0920852.

\end{document}